\documentclass[12pt]{article}
\usepackage{graphicx}
\usepackage{amsmath,amscd,amsthm,amsfonts,amssymb}
	\newtheorem{thm}{Theorem}
	
	\newtheorem*{mprobl}{Main Problem}
	\newtheorem{ass}{Assumption}
	
	\newtheorem{lem}{Lemma}
	
	\newtheorem{slem}{Sublemma}
	
	\newtheorem{prop}{Proposition}
	\newtheorem{defn}{Definition}
	
	\newtheorem*{ques}{Question}
	
	\newtheorem*{quest*}{Question}
	\newcounter{constant}

	\title{Minimal tori in $\mathbb{R}^4$}
	\date{ }
	\author{Marc Soret and Marina Ville}
	\begin{document}
		\maketitle
		\begin{abstract} We describe tools for the study of minimal surfaces in $\mathbb{R}^4$; some are classical (the Gauss maps) and some are newer (the link/braid/writhe at infinity). We illustrate them by desingularizing the Enneper surface in $\mathbb{R}^3$. Then we look for complete proper non holomorphic
			minimal tori with total curvature $-8\pi$ and a single end immersed in $\mathbb{R}^4$.  We translate the problem into a system of $10$ quadratic or linear equations in $11$ real variables with coefficients in terms of the Weierstrass function $\wp$ and give explicit solutions for these equations if $T$ is a rectangular torus. For the square torus, we have a complete answer with a unique family of solutions generalizing the Chen-Gackstetter torus in $\mathbb{R}^3$. On the other hand, we show that there is no solution on the equianharmonic torus.
		\end{abstract}
		
		\section{Introduction}
	\subsection{Statement of the problem}
		A minimal surface $S$ in $\mathbb{R}^4$ is a locally area minimizing surface; equivalently it  admits local parametrizations which are harmonic and conformal.\\
		The surface $S$ inherits a metric from the Euclidean metric in $\mathbb{R}^4$. If this metric has finite total (Gaussian) curvature and if $S$ is properly immersed in $\mathbb{R}^4$, then, outside some large ball in $\mathbb{R}^4$, $S$ consists in a finit union of {\it ends}. An {\it end} is an annulus $E$ parametrized as follows ($N\geq 1$ is an integer): 
		$$\{|z|>R\}\longrightarrow\mathbb{R}^4$$
		\begin{equation}\label{la parametrisation du bout}
		 z\mapsto(z^N+o(|z^N|),o(|z^N|))
		\end{equation}
		
		\begin{mprobl}\label{main problem}
			Given a $2$-torus $T$, find a proper minimal immersion from $T$  with one point removed, having one single end and total curvature $-8\pi$.
			\end{mprobl}
		A $2$-torus $T$ is the quotient of $\mathbb{C}$ by a lattice generated by two linearly independent complex numbers $\omega_1$ and $\omega_2$. The answer to the Main Problem depends on the conformal structure on $T$  which is determined by the quotient $\tau=\dfrac{\omega_2}{\omega_1}$.
		
		In $\mathbb{R}^3$, there is a unique minimal torus, namely the Chen-Gackstatter torus ([2]), with total curvature $-8\pi$ and a single end; it is parametrized by the square torus given by $\tau= i$. \\
		\\
		Using the implicit function theorem, Xie and Ma showed the existence of solutions to the Main Problem in $\mathbb{R}^{4}_1$ for tori close to the square torus  ([10]); moreover they point out that for every complex number $\lambda$, the cubic curve in $\mathbb{C}^2$ 
		\begin{equation}\label{algebraic torus}
		C_\lambda=\{(x,y)\slash y^2=x(x-1)(x-\lambda)\}
		\end{equation} 
		is a minimal surface in  $\mathbb{R}^4$ (identified with $\mathbb{C}^2$) of total curvature $-8\pi$. The single end of $C_\lambda$ occurs for $x$ and $y$ tending to infinity.  As $\lambda$ goes through the complex numbers, the $C_\lambda$'s describe all conformal types of the $2$-torus. We  will look here for explicit non holomorphic solutions of the Main Problem.
		\subsection{The tools}\label{the tools}
		 A minimal surface in $\mathbb{R}^4$ is locally parametrized by $4$ holomorphic functions $e,f,g,h$ by a harmonic map
		\begin{equation}\label{les quatre fonctions en general}
			z\in\mathbb{D}\mapsto \big(e(z)+\bar{f}(z),g(z)+\bar{h}(z)\big)\in\mathbb{C}^2\simeq\mathbb{R}^4
		\end{equation}
		which is conformal i.e.
		\begin{equation}\label{les derivees des quatre fonctions}
			e'f'+g'h'=0
		\end{equation}	
	Without loss of generality, the point that we remove from the torus will be $0$. Thus $e',f',g',h'$ will be meromorphic functions on the torus with a common pole at $0$ and this  common pole gives us the end of the minimal torus. Note that  for a general Riemann surface we would be integrating $1$-forms but the torus has a non zero holomorphic $1$-form $dz$ so we deal directly with meromorphic functions.\\
	  The Weierstrass $\wp$-function on a torus $T$ is a  meromorphic function with a single pole of order $2$ in a fundamental domain and it serves as a building block for the meromorphic functions on $T$. So $e',f',g',h'$ will be given by expressions in $\wp$ and $\wp'$  which we require to verify (\ref{les derivees des quatre fonctions}). \\
	  We also have the period problem: by construction,  the derivatives of the coordinates given by (\ref{les quatre fonctions en general}) are  $2$-periodic functions but we also needs their antiderivatives to be $2$-periodic. In other words, we require 
	  \begin{equation}\label{les periodes 20231}
	  \int_{\gamma_i}e'+\overline{\int_{\gamma_i}f'}=0\ \ \ \ 
	  \int_{\gamma_i}g'+\overline{\int_{\gamma_i}h'}=0\ \ \ \ i=1,2
	  \end{equation}
	  where the $\gamma_i$'s are the fundamental cycles of the torus.\\
	  
		We end up with a system of quadratic equations in several complex unknowns and their conjugates; if the torus is defined by the lattice generated by $1,\tau$, the coefficients of these equations involve $\tau$, the coefficents $g_2$, $g_3$ of $\wp$ and the periods $\eta_1$, $\eta_2$ of the antiderivative $\zeta$ of $-\wp$. \\
	REMARK. 	Actually we have not one system of equations but two, which we call Type (I) and Type (II). However, the Type (I) system has solutions only if the invariants of the torus verify a specific equations. We have not succeeded in proving that such a torus does not exist but numerical tests strongly suggest that. \\
	\\
	The calculations require numerical estimates for which we use Sagemath.
		 
		\subsection{Summary of the results}
		We begin with prelimineraries about a minimal surface $S$ in $\mathbb{R}^4$ (\S \ref{preliminaires}). We recall the various definitions of the Gauss maps and check that they coincide. Next we define the link/braid/writhe at infinity which gives us the degree of the normal bundle of $S$ if $S$ is embedded. We illustrate this by desingularizing the Enneper surface in $\mathbb{R}^4$ (\S \ref{enneper}). \\
	
		We deform the Chen-Gackstatter $3D$ torus in a complex dimension $1$ family of minimal square tori in $\mathbb{R}^4$ and these are the only solutions of the Main Problem on the square torus (Theorem \ref{theorem sur le square torus}). \\
		We translate the Main Problem in a system of real algebraic equations. This enable us to prove the unicity on the square torus. On the other hand, we derive that there are no solutions to the Main Problem for the equianharmonic torus (Proposition \ref{prop equianharmonic}). Finally 
	after making additional assumptions we construct explicit solutions of the Main Problem for all rectangular tori, i.e. $\tau=Ri$ for a positive number $R$ (Theorem \ref{thm:rectangular tori}). Finally we describe the ends of these minimal tori and derive that the $4D$ Chen-Gacksatter tori are not embedded. We conclude with the following question.
	\begin{ques}
	Are the rectangular minimal tori described in Theorem \ref{thm:rectangular tori} embedded?	
		\end{ques}
		
	\section{Preliminaries: minimal surfaces in $\mathbb{R}^4$}\label{preliminaires}
	We let $\Sigma$ be a properly immersed oriented minimal surface in $\mathbb{R}^4$ and assume that $\Sigma$ has finite total curvature (for the metric induced by the metric in $\mathbb{R}^4$).  

	\subsection{The Gauss maps}
	One defines two Gauss maps $\gamma_{\pm}:\Sigma\longrightarrow\mathbb{R}^4$.  We recall three equivalent definitions as each one has different advantages.\\
	Let $p\in\Sigma$ and let $(u,v)$ a positive orthonormal basis of  the tangent plane $T_p\Sigma$.
	\subsubsection{The Eells-Salamon approach ([4])}\label{eells-salamon subsection}
	Eells-Salamon identify the set of complex structures on $\mathbb{R}^4$ which preserve the metric and preserve (resp. reverse) the orientation of $\mathbb{R}^4$ with the $2$-sphere $Z_+$ (resp. $Z_-$) of unit $2$-vectors $\alpha$ with $\star \alpha=\alpha$ (resp. $\star \alpha=-\alpha$). The map $\star:\Lambda^2(\mathbb{R}^4)\longrightarrow \Lambda^2(\mathbb{R}^4)$ is the Hodge operator.  We set $\gamma_\pm(p)$ to be the unique such complex structure $J$ such that $T_p\Sigma$ is an oriented complex $J$-line. In other words, $J(u)=v$.\\
	Note that this definition does not depend on the choice of a basis of $\mathbb{R}^4$, which enabled [4] to extend it to a general Riemannian manifold, thus creating the {\it twistor approach} to minimal surfaces.\\
	\\ By contrast, the next two definitions depend on the choice of a positive orthonormal basis $(e_1,e_2,e_3,e_4)$ of $\mathbb{R}^4$.
	\subsubsection{The quaternions approach ([1])}\label{definition par les quaternions}
	We identify $\mathbb{R}^4$ with the quaternions $\mathbb{H}$
	by setting\\ $x_1e_1+x_2e_2+x_3e_3+x_4e_4\mapsto x_1+x_2i+x_3j+x_4k$ and we set 
	\begin{equation}
	\label{quaternions}
	\gamma_+(p)=v\cdot u^{-1}
	\end{equation}
	It is a unit imaginary quaternion so we can write it as $\gamma_+(p)=ai+bj+ck$ and identify it with the complex structure $J$ in $Z_+$ such that $Je_1=ae_2+be_3+ce_4$. 
	\subsubsection{The classical approach ([6])}\label{classique} 
	We complexify $\mathbb{R}^4$, write $(w_1,w_2,w_3,w_4)=u-iv\in 
	\mathbb{R}^4\otimes\mathbb{C}$ and set
	\begin{equation}
	\label{la forme de osserman et al}
	\gamma_{+}(p)=\frac{w_3+iw_4}{w_1-iw_2}\in\mathbb{C}P^1\ \ \ 
	\gamma_{-}(p)=\frac{-w_3+iw_4}{w_1-iw_2}\in\mathbb{C}P^1
	\end{equation}
	Thus, for the local expression (\ref{les quatre fonctions en general}), we have 
	\begin{equation}
	\label{la forme de osserman et al}
	\gamma_{+}(p)=\frac{g'}{f'}\ \ \ 
	\gamma_{-}(p)=-\frac{h'}{f'}
	\end{equation}
	
	We can see (\ref{la forme de osserman et al}) as a stereographic projection of (\ref{quaternions})
	\begin{lem}\label{lemme pour l'identification}
		If $u\cdot v^{-1}=ai+bj+ck$ (\S \ref{definition par les quaternions}), the $w_i$'s defined in \S \ref{classique} verify
		\begin{equation}\label{egalite entre les deux gauss maps1}
		\frac{w_3+iw_4}{-w_1+iw_2}=-i\frac{b+ic}{1-a}
		\end{equation}
	\end{lem}
	\begin{proof}
		Fix $J$ by setting $J(e_1)=ae_2+be_3+ce_4$. A computation shows that the $(w_1,w_2,w_3,w_4)$'s given by $e_k-iJe_k$ for $i=1,2,3,4$ all verify (\ref{egalite entre les deux gauss maps1}). Hence (\ref{egalite entre les deux gauss maps1}) is true for every $X-iJX, X\in \mathbb{R}^4$.
	\end{proof}
We derive the following from (\ref{la forme de osserman et al}) and from the Eells-Salamon construction  \S \ref{eells-salamon subsection}.
\begin{prop}\label{pas complexe}
	The tori constructed in Theorems \ref{theorem sur le square torus} and \ref{thm:rectangular tori} are not complex for any isometric complex structure on $\mathbb{R}^4$.
\end{prop}

	{\bf From now on, we  assume that $\Sigma$\ has finite total curvature i.e. \\ $\boxed{-\infty<\int_\Sigma K^T\leq 0}$ where $K^T$ is the curvature of the metric on $\Sigma$ induced by the metric on $\mathbb{R}^4$}.
	\subsection{The link/writhe/braid at infinity}\label{writhe  at infinity}
	For $R$ large enough, $\Sigma\cap(\mathbb{R}^4\backslash \mathbb{B}(0,R))$ is a finite union of ends (see (\ref{la parametrisation du bout})) above.
	The plane defined by the first complex coordinate in (\ref{la parametrisation du bout}) is the {\it tangent plane at infinity} of the end.
	The quantity $N-1$ is called the {\it branching order at infinity} for the end and $N$ is the {\it order of the end}.
	
	\subsubsection{Link and writhe at infinity}\label{link and writhe}
	If moreover $\Sigma$ is embedded outside a ball in $\mathbb{R}^4$, we define a link at infinity similarly to the case of complex curves (cf. [11]).\\  We let $\Sigma_R=\Sigma\cap \mathbb{B}(0,R)$; then, for $R>0$ large enough,
	\begin{equation}
	L_R=\partial\Sigma_R=\Sigma\cap \mathbb{S}(0,R)
	\end{equation}
	is a link whose isotopy type does not depend on $R$. We call it the {\it link at infinity} of $\Sigma$. \\
	Each component of $L_R$ corresponds to an end $E_i$ of $\Sigma$. For each $E_i$, take a vector $X_i$ in the normal plane at infinity to $E_i$. Formula (\ref{la parametrisation du bout}) shows us that on $E_i$, $X_i$ is close to being tangent to  $\mathbb{S}(0,R)$; thus we can project the $X_i$'s to $\mathbb{S}(0,R)$ 	and get a framing $X$ of $L_R$. If we push $L_R$ in the direction of $X_i$, we get another link $\hat{L}_R$. The linking number
	$lk(L_R,\hat{L}_R)$ is the {\it self-linking number} of $L_R$ w.r.t. the framing $X_E$.
	\begin{defn}
		For $R$ large enough, the {\it self-linking number} of $L_R$ w.r.t. the framing $X$ does not depend on the vector field $X$ orthogonal to the plane tangent at infinity at each end $E_i$. It also does not depend on $R$ and 
		we call it the {\bf writhe at infinity} $w_\infty(\Sigma)$ of $\Sigma$.
	\end{defn}  
The writhe at infinity is similar to the {\it self-linking number} defined by Joel Fine in ([5]). 
	\subsubsection{Braids}
	For a single end, the knot at infinity can be interpreted as a braid.\\
	We recall that a {\it closed braid} in $\mathbb{R}^3$ with oriented axis $Oz$ is a loop $\gamma(t)$ whose cylindrical coordinates $(\rho(t), \theta(t), z(t))$ verify for all $t$,
	$$\rho(t)\neq 0, \ \ \ \ \theta'(t)>0$$
	The number of strands of $\gamma$ is the degree of $\theta:\mathbb{S}^1\longrightarrow \mathbb{S}^1$. The algebraic length $e(\gamma)$ is the linking number of $\gamma$ with a loop $\hat{\gamma}$ obtained by pushing $\gamma$ slightly in the direction of $Oz$.\\
	By replacing a line by a great circle, we have a similar definition in $\mathbb{S}^3$.\\
	We can also write a braid with $N$ strands as an element of the braid group $B_N$ generated by the $\sigma_i$'s which exchange the $i$-th strand with the $i+1$-th strand. Writing  a braid as $\prod_{i} \sigma_{k_i}^{m(i)}$, we get
	\begin{equation}
	e(\prod_{i} \sigma_{k_i}^{m(i)})=\sum_{i}m(i)
	\end{equation}
	
	\subsubsection{A single end: the braid at infinity}
	Assume that $\Sigma$ has a single end. \\
	We let $(z_1,z_2)$ be complex coordinates in $\mathbb{R}^4\cong\mathbb{C}^2$.
	We defined $K_R$ in $\mathbb{S}(0,R)$ but we can define it equivalently in a cylinder $C(R)$ as
	$$K_R=\Sigma\cap\underbrace{\{(z_1,z_2)\in\mathbb{C}^2: |z_1|=R\}}_{C(R)}$$
	Thus, the knot $K_R$ is a closed $N$-braid with axis parallel to the vector $X_E$ and the writhe $w_\infty(E)$ is its algebraic length.
	\subsubsection{Several ends}\label{several ends}
	If $\Sigma$ has several ends $E_i$, $i=1,...,k$, each end has a knot at infinity which is a closed braid.  We can define the writhe at infinity of each end $w_\infty(E_i)$ as in \S \ref{link and writhe}. 
	\begin{prop} If $\Sigma$ has $k$ ends, each of branching order $N_i-1$ and if the tangent planes at infinity $P_i$ are mutually transverse, then 
		\begin{equation}
		w_\infty(\Sigma)=\sum_{i=1}^k w_\infty(E_i)+2\sum_{1\leq i<j\leq k}N_iN_j\sigma(i,j)
		\end{equation}
		where $\sigma(i,j)$ is $1$ (resp. $-1$) if $P_i$ and $P_j$ intersect positively (resp. negatively).
		\end{prop}
	\subsubsection{Example: the ends of the $4D$ Enneper surfaces}\label{enneper}
	An Enneper surface is a minimal immersion of $\mathbb{R}^2$ in $\mathbb{R}^4$ with a codimension $1$ singularity and total curvature $-4\pi$. It is is parametrized by
	\begin{equation}\label{enneper3d} z\mapsto (\dfrac{1}{3}z^3-\bar{z},
\dfrac{1}{2}z^2+\dfrac{1}{2}\bar{z}^2)
	\end{equation}  In $\mathbb{R}^4$, we  deform it in two different ways
	\begin{equation}\label{enneper4d points doubles} 
		i)\ \ \ \ \ \ \ \ \ \ \ \ z\mapsto (\dfrac{1}{3}z^3-\bar{z},
	\dfrac{\lambda}{2}z^2+\dfrac{1}{2\lambda}\bar{z}^2)\ \ \ \ \ \ \ \lambda\in\mathbb{C}
	\end{equation} which has two transverse double points. Its end is similar to the end of the complex algebraic curves (\ref{algebraic torus}) and its braid at infinity is the braid of the $(3,2)$-torus knot.\\
	\includegraphics[scale=0.4]{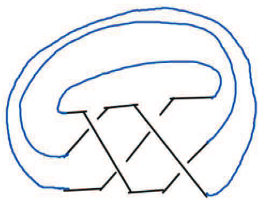}\\
	\begin{equation}\label{embedded4Denneper}
	ii)\ \ \ \ \ \ \ \ \ \ \ \  z\mapsto (\dfrac{1}{3}z^3-is^2z-\bar{z},
	\dfrac{1}{2}z^2+\dfrac{1}{2}\bar{z}^2+se^{\frac{i\pi}{4}}z-se^{-\frac{i\pi}{4}}\bar{z})
		\end{equation} which is embedded.  The knot at infinity is trivial\\
		\includegraphics[scale=0.4]{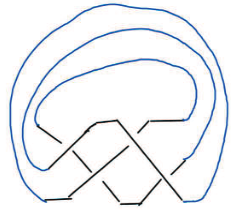}\\
		REMARK. At the moment, we have very few examples of  proper minimal non holomorphic embeddings of $\mathbb{C}$ of finite curvature in $\mathbb{R}^4$. Here
		is one of them.\begin{prop}
			The following is a minimal embedding of $\mathbb{C}$ in $\mathbb{R}^4$ of total curvature $-8\pi$. It is preserved by the symmetries w.r.t. the planes generated by $(e_1,e_3)$ and $(e_2,e_4)$.
			\begin{equation}\label{formule pour F}
			F:z\mapsto(\frac{z^5}{5}-z+\bar{z},-\frac{z^3}{3}+\frac{\bar{z}^3}{3}+z+\bar{z})
			\end{equation}
			\[ = \left( \begin{array}{c}
			\frac{u^5}{5}-2u^3v^2+uv^4 \\
			\frac{v^5}{5}+u^4v-2u^2v^3-2v\\
			2u\\
			\frac{2v^3}{3}-2u^2v \end{array} \right).\] 
		\end{prop}
	\subsection{Integral formulae for the  curvatures}\label{integral formulae for the normal bundle}
	We recall some classical properties ([6], [7], etc)).\\
	There exists a closed Riemann surface $\hat{\Sigma}$ without boundary and a finite number of points $p_1,...,p_d$ in 
	$\hat{\Sigma}$ such that 
	$\Sigma=\hat{\Sigma}\backslash\{p_1,...,p_d\}$
	and the Gauss maps $\gamma_{\pm}$ extends to holomorphic maps $\hat{\gamma}_\pm:\hat{\Sigma}\longrightarrow \mathbb{C}P^1.$ We denote by $d_+$ (resp. $d_-$) the degree of $\hat{\gamma}_+$ (resp. $\hat{\gamma}_-$). If $K^T$ and $K^N$ are the curvatures of the tangent and normal bundles, we have
	\begin{equation}\label{equation:the two curvatures}
	-\int_{\Sigma}K^T=2\pi (d_++d_-)\ \ \ \ \ -\int_{\Sigma}K^N=2\pi (d_+-d_-)
	\end{equation}
	REMARK. Depending on the conventions, the Gauss maps are holomorphic or antiholomorphic and here we take them to be holomorphic so $d_\pm\geq 0$. 
	\begin{thm}\label{thm: integrale de la courbure tangente}([7]) If  $\Sigma$ has $k$ ends, each of brancing order $N_i-1$ and Euler characteristic $\chi(\Sigma)$, then
		\begin{equation}
		\frac{1}{2\pi}\int_{\Sigma}K^T=-\sum_{i=1}^k N_i+\chi(\Sigma)
		\end{equation}
	\end{thm}
The next theorem and its proof use the notions and notations of \S \ref{writhe  at infinity}.
	\begin{thm}\label{proposition:curvature of the normal bundle}  
		If $\Sigma$ is embedded except for a finite number $D_\Sigma$ of transverse double points counted with sign,
		\begin{equation}\label{courbure du fibre normal;immerge}
		\frac{1}{2\pi}\int_{\Sigma}K^N=w_\infty(\Sigma) -2D_\Sigma
		\end{equation}

	\end{thm}
	\begin{proof} 
		The vector fields $X_i$ defined on $\partial\Sigma_R$ are very close to belonging to $N\Sigma$ so we can define a section $X^N$ of $N\Sigma_R$ which is very close to $X_i$ on each end $E_i$. We compute the number of zeroes $Z(X^N)$ in two different ways. \\
		1) If $p\in\mathbb{S}(0,R)\cap\Sigma$, notice that $X^N$ is close to $T_p \mathbb{S}(0,R)$. Thus we can push $\Sigma_R$ in the direction of $X^N$ to get a surface $\hat{\Sigma}_R$ in $\mathbb{B}(0,R)$ bounded by $\hat{L}_R$. We recall that the linking number of two links $L_1$ and $L_2$ is the number of intersection points of two surfaces, each of them bounded by one of the $L_i$'s. Thus $lk(L_R,\hat{L}_R)$ is the number of intersection points between $\Sigma_R$ and  $\hat{\Sigma}_R$. \\
		If we denote by $Z(X^N)$ the number of zeroes of $X^N$, this means
		\begin{equation}\label{z avec le w}
			Z(X^N)=w_\infty(\Sigma)-2D_\Sigma	
		\end{equation}
		2) 
		We let $J$ be the complex structure compatible with the metric and orientation on $N\Sigma$ and define a connection form.  $$\omega=-\frac{1}{\|X^N\|^2}<\nabla X^N, JX^N>.$$ 
		
		Note that $d\omega=K^NdA$, where $dA$ is the area element on $\Sigma$. We apply Stokes' theorem to $\omega$ on $\Sigma_R\setminus \cup_p B(p,\epsilon)$ where $p$ goes through the double points of $\Sigma_R$; we let $\epsilon$ tend to $0$ and get
		\begin{equation}\label{z avec stokes}
		\frac{1}{2\pi}\int_{\Sigma_R}K^N=\frac{1}{2\pi}\int_{L_R}\omega+Z(X^N)
		\end{equation}
		As $R$ becomes large, $X^N$ is asymptotic to the constant vector $X_i$ on each end $E_i$, thus
		$$\lim_{R\longrightarrow\infty}\int_{L_R}\omega=0$$
		Putting together (\ref{z avec le w}) and (\ref{z avec stokes}) yields the theorem.
	\end{proof}

		\section{Preliminaries: the torus}
		\subsection{A $2$-torus and its Weierstrass $\wp$-function}\label{section sur les tores}
		We let $\tau$ be a complex number  with $Im(\tau)>0$ and  
		consider the torus $T_\tau$ defined as the quotient of $\mathbb{C}$ by the lattice $\{m+n\tau\slash (m,n)\in\mathbb{Z}^2\}$.  We let $\wp$ be the Weierstrass function on $T_\tau$, namely
		$$\wp(z)=\frac{1}{z^2}+\sum_{(m,n)\neq(0,0)}\Big(\frac{1}{(z+m+n\tau)^2}-\frac{1}{(m+n\tau)^2}\Big).$$ 
		The meromorphic function $\wp$ is  doubly periodic for the periods $1$ and $\tau$ and has a single pole in a fundamental domain of $T_\tau$. \\
		We let $\gamma_1$ (resp. $\gamma_2$) be the $1$-cycle on $T_\tau$ corresponding to $1$ (resp. $\tau$) and introduce the periods of $\wp$, namely

		\begin{equation}
		\eta_1=-\int_{\gamma_1}\wp\ \ \ \ \ \ \ \ 
		\eta_2=-\int_{\gamma_2}\wp
		\end{equation}
		The Legendre relation connects $\tau, \eta_1$ and $\eta_2$
		\begin{equation}\label{relation de Legendre}
		\tau\eta_1-\eta_2=2\pi i
		\end{equation}
		Moreover we have
		\begin{equation}
		\label{equa diff de weierstrass}
		\wp'^2=4\wp^3-g_2\wp-g_3.
		\end{equation}
		We derive
		\begin{equation}\label{p seconde}
		\wp^2=\frac{\wp''}{6}+\frac{g_2}{12}
		\end{equation}
		This means that, for $i=1,2$
		\begin{equation}\label{integrale de p au carre}
		\int_{\gamma_1}\wp^2=\frac{g_2}{12}\ \ \ \ 
		\int_{\gamma_2}\wp^2=\frac{g_2}{12}\tau
		\end{equation}

	\subsection{The $4$ meromorphic functions which define the minimal tori: precisions and assumptions}\label{conditions sur les quatre fonctions}
	As explained in \S \ref{the tools}, the torus $T_\tau$ is parametrized by $(e+\bar{f},g+\bar{h})$ where $e,f,g,h$ have a common pole at $0$ and verify $e'f'+g'h'=0$.\\
	We have assumed that the total curvature of the surface is $\int_{T_\tau}K=-8\pi$.  On the other hand   the end of the surface is parametrized by\\ $z\mapsto(z^N+o(|z^N|),o(|z^N|))$ (cf. (\ref{la parametrisation du bout})) for some $N$; thus Jorge-Meeks formula ([7]) yields
	\begin{equation}
	-8\pi=\frac{1}{2\pi}\int_{T_\tau}K=\chi(T_\tau-\{0\})-N=-1-N
	\end{equation}
	So $N=3$ and the poles of $e',f',g',h'$ have order at most $4$.
	Without loss of generality, we assume

	\begin{ass}
		The function $e'$ has a pole of order $4$ whereas $g'$ and $h'$ have poles of order $\leq 3$.
	\end{ass}
Thus the identity $e'f'+g'h'=0$ ensures that $f'$ has a pole of order at most $2$ at $0$, unless it is constant, so on a fundamental domain $\Delta$ containing $0$, $f'$ is of the form $$f'(z)=\frac{\mu}{z^2}+\frac{\lambda}{z}+[z]$$
where $\lambda,\mu$ are complex numbers and $[z]$ is a power series in $z$.\begin{itemize}
	\item 
 If $\mu= 0$, then $f'$ is constant. 
 \item 
 If $\mu\neq 0$, $f'-\mu\wp$ has at most a single pole of order $1$, so it is a constant function and we have $$f'=\mu\wp+\phi$$
for some complex number $\phi$.
\end{itemize}
So, possibly after a change of coordinates in $\mathbb{C}^2$,  we can make the assumption
\begin{ass}\label{assumption: les deux types}
The function $f'$ is of one of the following types
\begin{itemize}
	\item 
	type (I): $f'=1$
	\item 
	type (II): $f'=\wp+s$, for some $s\in\mathbb{C}$
\end{itemize}
\end{ass}
We will show see that Type (I) is either impossible or at least very unusual.

		\section{The square torus}
		The {\it square torus} $T_i$ is given by periods $1$, $i$. Moreover
		\begin{equation}
			\eta_1=\pi\ \ \ \ \ \ \ \ 
			\eta_2=-i\pi\ \ \ \ \ \ \ \ g_3=0
		\end{equation}
			We deform the Chen-Gackstatter minimal square torus in $\mathbb{R}^3$ into a family of minimal tori in $\mathbb{R}^4$ and get
		\begin{thm}\label{theorem sur le square torus}
			Let $\wp$ be the Weierstrass function on the square torus and let $A$ be the constant defined by Chen-Gackstatter in dimension $3$, namely
			\begin{equation}\label{valeur de la constante A}
				A^2=\frac{3\pi}{2g_2}
			\end{equation} 
			\begin{enumerate}
				\item 
				We identify $\mathbb{R}^4$ with $\mathbb{C}^2$. For each non zero complex number $\lambda$, the following is a minimal map from the square torus $T_i\longrightarrow\mathbb{R}^4$
				\begin{equation}\label{4D tori}
					z\mapsto\Big( 
					A^2\int_{z_0}^z (4\wp^2(w)-g_2)dw-\overline{\int_{z_0}^z\wp(w)dw},\lambda A\wp(z)+\frac{\bar{A}}{\overline{\lambda}}\bar{\wp}(z) \Big)
				\end{equation}
				\begin{equation}\label{solution sur le tore carre}
					\mbox{i.e.}\ \ \ \ \ \ z\mapsto\Big( 
					\frac{2}{3}A^2(\wp'(z) -g_2z)+\bar{\zeta}(z),\lambda A\wp(z)+\frac{\bar{A}}{\overline{\lambda}}\bar{\wp}(z) \Big)
				\end{equation}
				where $\zeta$ is the Weierstrass zeta function.\\
				This is the Chen-Gackstatter torus inside $\mathbb{R}^3$ if and only if $|\lambda|=1$.
				
				\item 
				The maps (\ref{4D tori}) are the only non holomorphic proper  minimal immersions of the square torus in $\mathbb{R}^4$ with curvature $-8\pi$ and a single end. 
			\end{enumerate}
		\end{thm}
	We recall
		\begin{thm}\label{unicite de cg}([9],[13])
		The Chen-Gackstatter torus is the only minimal torus of total curvature $-8\pi$ and one end in $\mathbb{R}^3$.
		\end{thm}
	The existence (1. Th. \ref{theorem sur le square torus}) follows immediately from (\ref{les derivees des quatre fonctions}); note that the period conditions are trivially satisfied. The unicity (2. Th. \ref{theorem sur le square torus}) will be proved in \S \ref{unicity for the square torus}.\\

		\section{General case: a system of algebraic equations}
		We consider a torus $T_\tau$ as in \S \ref{section sur les tores} and a minimal map $F:T_\tau\longrightarrow\mathbb{C}^2$ of the form
		\begin{equation}\label{expression de F}
		F=(e+\bar{f},g+\bar{h})
		\end{equation}
		verifying the assumptions of \S \ref{conditions sur les quatre fonctions}; without loss of generality,  we now  add the following assumption to the list in \S \ref{conditions sur les quatre fonctions}:
		\begin{ass}\label{assumption on the second coordinate}
		The second coordinate of $F$ is not of the form $u\wp'+\bar{v}\overline{\wp'}$.
	\end{ass}
Indeed, if a surface is parametrized by $(e+\bar{f},u\wp+\bar{v}\overline{\wp})$, we introduce $w$ with $w^2=uv$; then  
		$(e+\bar{f},
		w\wp+\bar{w}\overline{\wp})$ is also a minimal torus with one pole and curvature $-8\pi$. Moreover, it sits inside $\mathbb{R}^3$ so it is the Chen-Gackstatter surface on the square torus (Theorem \ref{unicite de cg}).

\subsection{Type (I)}
In Assumption \ref{assumption: les deux types} above, we distinguished between types (I) and (II). We now investigate type (I) surfaces, i.e. those with $f'=1$.
\begin{prop}\label{les solutions de type (I)}

	For a given torus $T_\tau$, there is a type (I) solution to the Main Problem if and only $g_2\neq 0$ and 
	\begin{equation}\label{condition pour le type I}
	|g_2|=12|\eta_1\bar{\eta_1}-\frac{\pi}{Im\tau}(\eta_1+\bar{\eta_1})|
	\end{equation}

	\end{prop}
We do not know if there is any torus satisfying (\ref{condition pour le type I}); if there are any, they are not generic.

\begin{proof}
		We look for $e',f',g',h'$ of the form 
	$$e'=S\wp^2+T\wp'+U\wp+V\ \ \ f'=1\ \ \ g'=\lambda\wp'+\alpha\wp+\beta \ \ \ h'=\mu\wp'+a\wp+b$$
	where $S,T,U,V,\lambda,\mu,a,b,\alpha,\beta$ are complex numbers and $S\neq 0$.\\
	We write the identity $e'f'+g'h'=0$ and get rid of the term in $\wp'^2$ by using 
	(\ref{equa diff de weierstrass}). Then we identify to zero the coefficients of $1$, $\wp$, $\wp'$, $\wp^2$, $\wp^3$, $\wp\wp'$ and derive the system
	\begin{equation}\label{les grandes lettres}
	\left  \{\begin{array}{c}
		\lambda=\mu=0\\
	S+\alpha a=0\\
	T=0\\
	U+\alpha b+\beta a=0 \\
	V+\beta b=0
	\end{array} \right.
	\end{equation}
	The period equations (\ref{les periodes 20231}) for the first components
	\begin{equation}\label{period first componente petit}
	S\frac{g_2}{12} -U\eta_1+V+1=0\ \ \ \ \ \ \ \ \ \
	S\frac{g_2}{12}\tau-U\eta_2+V\tau+\bar{\tau}=0
	\end{equation}
	
	Using Legendre's equation (\ref{relation de Legendre}) we rewrite the second equation of (\ref{period first componente petit}) as
	\begin{equation}
	\label{deuxieme forme de 26}
(S\frac{g_2}{12} -U\eta_1+V)\tau+2\pi i U+\bar{\tau}=0
	\end{equation}
	We derive $U=\dfrac{1}{\pi}Im\tau$; and using (\ref{les grandes lettres}), we rewrite (\ref{period first componente petit}) 
	\begin{equation}\label{periode;la premiere composante;bis}
	\left  \{\begin{array}{c}
	-\alpha a\frac{g_2}{12}-\beta b- \frac{Im\tau}{\pi}\eta_1+1=0\\
	\\
	\alpha b+\beta a=-\frac{Im\tau}{\pi}
	\end{array} \right.
	\end{equation}
	The period of the second component is 
	\begin{equation}\label{deuxieme periode; petit}
	(\star)\ \ \  -\alpha\eta_1+\beta-\bar{a}\bar{\eta_1}+\bar{b}=0\ \ \ \ \ (\star\star)-\alpha\eta_2+\beta\tau-\bar{a}\bar{\eta_2}+\bar{b}\bar{\tau}=0
	\end{equation}
	$$\mbox{i.e.}\ \ \ (\beta-\alpha\eta_1)(\tau-\bar{\tau})=2\pi i(\bar{a}-\alpha)\ \ \ 
	(\bar{a}\bar{\eta_1}-\bar{b})(\tau-\bar{\tau})=2\pi i(\bar{a}-\alpha)$$
	so we rewrite (\ref{deuxieme periode; petit}) as\ \ \ \ 
	$\label{beta et mu}
	\left  \{\begin{array}{c}
	\beta=\Big(\frac{\pi}{Im\tau}\Big)\bar{a}+\Big(\eta_1-\frac{\pi}{Im\tau}\Big)\alpha\\
	b=\Big(\frac{\pi}{Im\tau}\Big)\bar{\alpha}+\Big(\eta_1-\frac{\pi}{Im\tau}\Big)a
	\end{array} \right.
	$\\ 
	
We plug these expressions of $b$ and $\beta$ into (\ref{periode;la premiere composante;bis})
	and  derive
	\begin{equation}\label{deux equations en alpha et a}
	\left  \{\begin{array}{c}
	(1)\ \ \ \ \ \  \frac{\pi}{Im\tau}(|a|^2+|\alpha|^2)+
	2(\eta_1-\frac{\pi}{Im\tau})a\alpha+\frac{Im\tau}{\pi}=0 
	\\
	\\
	(2)\ \ \ \ \  a\alpha[-\frac{g_{2}}{12}+(\eta_1-\frac{\pi}{Im\tau})^2]-(\frac{\pi}{Im\tau})^2\bar{a}\bar{\alpha}=0 \end{array} \right.
	\end{equation}
We let
$a\alpha=|a\alpha|e^{i\theta}$, 
and derive from (1) of (\ref{deux equations en alpha et a}) that $(\eta_1-\frac{\pi}{Im\tau})e^{i\theta}$ is a negative real number. Thus (2) in (\ref{deux equations en alpha et a}) yields
$\dfrac{g_{2}}{12}e^{2i\theta}=|\eta_1-\dfrac{\pi}{Im\tau}|^2-(\dfrac{\pi}{Im\tau})^2$ and the lemma follows.\ \ \end{proof}

\subsection{Type (II)}
\begin{prop}\label{thm type II 2023}
	The torus $T$ admits a type (II) solution to the Main Problem if there are complex numbers $a,b,c,d,s,t,u,v,w,y,z$
	which verify the folowing systems of equations
		\begin{equation}\label{premier systeme avec six}
		\left\{ \begin{array}{c}
			
			a+4tw=0\\
			b+ty+uw=0\\
			c+as+uy=0\\
			d+cs+uz+vy-g_2tw=0\\
			bs + vw + tz=0\\
			ds  + vz-g_3tw=0
		\end{array} \right.
	\end{equation}
\begin{equation}\label{periode du premier 2023. bis1}
	\left\{ \begin{array}{c}
		c-1+\dfrac{Im\tau}{\pi}(\bar{\eta}-\bar{s})=0\\
		c(1-\dfrac{Im\tau}{\pi}\eta)+\dfrac{Im\tau}{\pi}(a\dfrac{g_2}{12}+d)-1=0
	\end{array} \right.
\end{equation}
\begin{equation}\label{periode du second 2023. bis1}
	\left\{ \begin{array}{c}
		u+\bar{y}(\dfrac{Im\tau}{\pi}\bar{\eta}-1)-\dfrac{Im\tau}{\pi}\bar{z}=0\\
		u(1-\dfrac{Im\tau}{\pi}\eta)+\dfrac{Im\tau}{\pi}v-\bar{y}=0
	\end{array} \right.
\end{equation}
	\end{prop}
REMARK. We have $10$ equations in $11$ unknowns: it is easy to eliminate  $a,b,s,d,u,y$ or $a,b,c,d,u,y$ and get $4$ equations in $5$ unknowns.
\begin{proof}
	Consider a torus $T_\tau$ with parameters $\tau$, $g_2$, $g_3$, $\eta_1$ (see \S \ref{section sur les tores}) and let
	\begin{equation}
		e'=a\wp^2+b\wp'+c\wp+d\ \ \ \ \ f'=\wp+s\ \ \ \ \ g'=t\wp'+u\wp+v\ \ \ \ h'=w\wp'+y\wp+z
	\end{equation}
	We write $e'f'+g'h'=0$ and use (\ref{equa diff de weierstrass}). We get an expression in $\wp^3$, $\wp\wp'$, $\wp^2$, $\wp'$, $\wp$ and $1$. These functions have poles of different orders at $0$ so the coefficient of each of them in $e'f'+g'h'=0$ is zero and we derive the system (\ref{premier systeme avec six}).\\
			Next, we consider the equations (\ref{les periodes 20231}) given by the periods.
	Using Legendre's equation, we write for the first coordinate
	\begin{equation}\label{periode du premier 20231}
		\left\{ \begin{array}{c}
			a\dfrac{g_2}{12}-c\eta_1+d-\overline{\eta_1}+\bar{s}=0\ \  \ \ \ \ \ \ \ \ \ \ (\star)\\
			a\dfrac{g_2}{12}\tau-c(\tau\eta_1-2i\pi)+d\tau-
			(\bar{\tau}\overline{\eta_1}+2i\pi)+\bar{s}\bar{\tau}=0 \ \  \ \ \ (\star\star)
		\end{array} \right.
	\end{equation}
	Computing $(\star\star)-\tau(\star)$ and $(\star\star)-\bar{\tau}(\star)$, we derive
(\ref{periode du premier 2023. bis1}).\\
		We proceed similarly for the second component and derive
	(\ref{periode du second 2023. bis1}). \ \ \end{proof}
\subsection{Unicity of the solutions for the square torus}\label{unicity for the square torus}
We prove here 2. of Theorem \ref{theorem sur le square torus}. Clearly there is no type (I) solution for the square torus so we look for type (II) solutions. 
 We derive from (\ref{periode du premier 2023. bis1}) and (\ref{periode du second 2023. bis1})
	\begin{equation}\label{tore carre.premierii}
	\left\{ \begin{array}{c}
		(1)\ \  s=\pi\bar{c}\\
		\ \ \ \ \ \ \ 	(2)\ \ d+\dfrac{g_2}{12}a=\pi\\
		(3)\ \  z=\pi\bar{u}\\
		(4)\ \  v=\pi\bar{y}
	\end{array} \right.
\end{equation} 
We plug in (1), (3) and (4) of (\ref{tore carre.premierii}) into the $4$-th equation of (\ref{premier systeme avec six}). We derive that $d+\dfrac{g_2}{4}a$ is real; hence $a$ and $d$ are both real (see (2) of (\ref{tore carre.premierii}) and recall that $g_2$ is real).  . \\
Using (\ref{tore carre.premierii}), we rewrite the $3$-rd and $6$-th equations of (\ref{premier systeme avec six}) as $c+a\pi\bar{c}+uy=0$ and $dc+\pi uy=0$, hence $(d-\pi)c=a\pi^2\bar{c}$. Thus, if $c\neq 0$, $d-\pi=\pm a\pi^2$ and $\pi^2=\pm \dfrac{g_2}{12}$, which is false (since $a\neq 0$).\\
We derive that $s=vz=uy=0$.  \\
Since $t$ and $w$ are not zero, we derive from the $5$-th equation in (\ref{premier systeme avec six}) that $v=z=0$. Thus we get $u=y=0$, and finally $b=0$ and we have the $4D$-Chen-Gackstatter torus. \ \ \ \ \ \qed
\section{Equianharmonic torus}
\begin{prop}\label{prop equianharmonic}
	Let $T_\tau$ be the equianharmonic torus, i.e. with $\tau=e^{\frac{2\pi i}{3}}$. There exists no minimal immersion of $T_\tau$ into $\mathbb{R}^4$ with one end and total curvature $-8\pi$. 
\end{prop}
\begin{proof}
	We recall that the invariants of $T_\tau$ are $Im\tau=\dfrac{\sqrt{3}}{2}$,
	$g_2=0$, $g_3=1$ and $\eta_1=\dfrac{2\pi}{\sqrt{3}}$ (see [12]).
	Clearly they do not satisfy the requirements for Type (I) so we investigate type (II). Equations (\ref{periode du premier 2023. bis1}) and (\ref{periode du second 2023. bis1}) give us
	\begin{equation}\label{tore carre.periodes}
		c=\dfrac{Im\tau}{\pi}\bar{s}\ \ \ \ \ \
			d=\dfrac{\pi}{Im\tau}\ \ \ \ \ \ 
			u=\dfrac{Im\tau}{\pi}\bar{z}\ \ \ \ \ \ 
		v=\dfrac{\pi}{Im\tau}\bar{y}
	\end{equation}
So the $4$-th equation of (\ref{premier systeme avec six}) yields
$\dfrac{Im\tau}{\pi}+\dfrac{Im\tau}{\pi}s\bar{s}
+\dfrac{Im\tau}{\pi}z\bar{z}+
\dfrac{\pi}{Im\tau}y\bar{y}=0$,
which is not possible.\ \ \end{proof}

\section{Rectangular tori}
\begin{thm}\label{thm:rectangular tori}
	Let $T_{Ri}$ be a rectangular torus, i.e. it is given by a lattice generated by $1,Ri$ with $R$ a positive real number. There exists  a non holomorphic proper minimal map ${\mathcal T}_R$ from  $T_{Ri}$ into $\mathbb{R}^4$ with one end and total curvature $-8\pi$. Its coordinates are rational functions in terms of 
	\begin{itemize}
		\item 
		$R$
		\item 
		the coefficients $g_2$, $g_3$ of the Weierstrass $\wp$-functions
		\item 
		the quasi-period $\eta_1$ of the $\zeta$ function on the 
		period $1$.
	\end{itemize}
	
\end{thm}
The explicit formulae are in \S \ref{Explicit formulae}.

\subsection{Preliminaries; sketch of the proof}
 For a rectangular torus, the quantities $\eta_1$, $g_2$ and $g_3$ are real which makes the equations of Proposition \ref{thm type II 2023} easier to handle. To solve them, we make the following extra assumptions.
\begin{ass}
	\begin{enumerate}
		\item 
		The unknowns $a,b,c,d,s,t,u,v,y,z$ are real.
		\item 
		The functions giving the second component are of the form\\ $g'=t\wp'+u\wp+v$ and $h'=t\wp'-u\wp-v$, that is, the unknowns verify
		\begin{equation}
			t=w, u=-y, v=-z.
		\end{equation}
	\end{enumerate}
\end{ass}

Since the lattices of periods $(1,Ri)$ and $(1, \frac{1}{R}i)$ are conformally equivalent, we assume, this time without loss of generality:
\begin{ass} 
$R> 1$
\end{ass}
We rewrite the equations of Prop. \ref{thm type II 2023} in our current case. We eliminate  $a,b,s,d$ and we get a system of $3$ algebraic equations in the $3$ unknowns $c,t,u$.
After some computations, we are able to boil down this system of equations to a single linear equation in the unknown $u^2$. It is of the form 
\begin{equation}\label{equation resumee}
A(g_2,g_3,\eta,R)=u^2B(g_2,g_3,\eta,R)
\end{equation}
where $A$ and $B$ are rational fractions in $g_2,g_3,\eta,R$.\\ 
First we show that the expression $B(g_2,g_3,\eta,R)$ is always negative. It follows that the equation (\ref{equation resumee}) has a solution in $u$ if and only if 
\begin{equation}\label{equation resumee1}
A(g_2,g_3,\eta,R)<0
\end{equation}
The bulk of the proof will consist in proving (\ref{equation resumee1}) for every $R>1$. We use the reprentation of $g_2$, $g_3$ and $\eta$ as series in 
\begin{equation}\label{definition de q de R}
q(R)=e^{-2\pi R}
\end{equation}
 We consider three separate cases:
\begin{enumerate}
	\item $R>1.15$ and we estimate $A$ in terms of $q$
	\item 
	$1<R\leq 1.05$. If $R=1$, $A$ is not defined but it various terms are and we approximate these various parts as factors of $R-1$.
	\item 
	$1.05\leq R\leq 1.15$. We do numerical estimates (using Sagemath) on $10$ subintervals of $[1.05,1.15]$.
\end{enumerate} 

\subsubsection{The series expressing $g_2$, $g_3$ and $\eta$}\label{section avec les series de lang}
 Following [8] we write (see (\ref{definition de q de R}) above for the definition of $q(R)$). 
\begin{equation}
\eta_1(R)=\frac{\pi^2}{3}\Big[1-24\sum_{n=1}^\infty\frac{nq(R)^n}{1-q(R)^n}\Big]
\end{equation}
\begin{equation}
g_2(R)=\frac{4\pi^4}{3}\Big[1+240\sum_{n=1}^\infty\frac{n^3q(R)^n}{1-q(R)^n} \Big]\ \ \ 
g_3(R)=\big(\frac{2}{3}\big)^3\pi^6\Big[1-504\sum_{n=1}^\infty\frac{n^5q(R)^n}{1-q(R)^n} \Big]
\end{equation}     
So they convergence very fast to the following limits.        
\begin{equation}\label{les trois limites}
\lim_{R\longrightarrow+\infty}\eta_1(R)=\frac{\pi^2}{3}\ \ 
\lim_{R\longrightarrow+\infty}g_2(R)=\frac{4\pi^4}{3}\ \ 
\lim_{R\longrightarrow+\infty}g_3(R)=\big(\frac{2}{3}\big)^3\pi^6
\end{equation}

	\subsection{The condition for the existence of solutions - preliminary expression}
	NOTATION. We drop the subscript $1$ and write $\eta$ for $\eta_1$.\\
	We go back to the equations of Proposition \ref{thm type II 2023}. The equations (\ref{periode du premier 2023. bis1}) and (\ref{periode du second 2023. bis1}) are equivalent to\ \ \ \
	$\left\{ \begin{array}{c}
			s=\eta+\dfrac{\pi}{R}(c-1)\\
	d=c(\eta-\dfrac{\pi}{R})+\dfrac{\pi}{R}+\dfrac{t^2}{3}g_2\\
	v=u(\eta-\dfrac{2\pi}{R})
	\end{array} \right.$\\
	 The system (\ref{premier systeme avec six}) becomes (note that $a=-4t^2$)
	\begin{equation}\label{les cinq identites avec vbis}  \left  \{\begin{array}{c}
(1)\ \ 	4t^2[(\dfrac{\pi}{R}-\eta)-\dfrac{\pi}{R}c]+c-u^2=0\\
\\
(2)\ \ 	\dfrac{\pi}{R}c^2-\dfrac{2}{3}g_2t^2+
2(\eta-\dfrac{\pi}{R})c
+2(\dfrac{2\pi}{R}-\eta)u^2
+\dfrac{\pi}{R}=0\\
\\
(3)\ \ 	[(\eta-\dfrac{\pi}{R})+\dfrac{\pi}{R}c][g_2\dfrac{t^2}{3}+(\eta-\dfrac{\pi}{R})c+\dfrac{\pi}{R}]-(\dfrac{2\pi}{R}-\eta)^2u^2-g_3t^2=0
\end{array} \right.\end{equation}
We compute for the equations of (\ref{les cinq identites avec vbis})
$$0=\frac{g_2}{3}(1)+4(\frac{\pi}{R}-\eta)(2)+4(3)$$
$$\mbox{and derive}\ \ \ \ \ 0=t^2[\frac{8}{3}(\eta-\frac{\pi}{R})g_2-4g_3]+[4\eta(\frac{2\pi}{R}-\eta)+\frac{g_2}{3}](c-u^2).\ \ \ \  \mbox{Thus}$$
\begin{equation}\label{v en fonction de X et gamma}
t^2=(c-u^2)T(R)
\end{equation}
\begin{equation}\label{definition de T}
\mbox{where}\ \ \ \ \ \ \ T(R)=\frac{R(g_2-12\eta_1^2)+24\eta_1\pi}{R(12g_3-8g_2\eta_1)+8g_2\pi}
\end{equation}

We plug (\ref{v en fonction de X et gamma}) into Equation (1) of (\ref{les cinq identites avec vbis}) above 
and derive 
\begin{equation}\label{equation un reecrite}
(c-u^2)\Big(4[\eta-\frac{\pi}{R}+\frac{\pi}{R}c]T(R)-1\Big)=0
\end{equation}
\begin{lem}
	There is no solution with $c=u^2$. 
\end{lem}
\begin{proof}
	If $u^2=c$, then $t=0$ and (2) becomes  $\dfrac{\pi}{R}(u^4+2u^2+1)=0$ which has no real solutions.
\end{proof}

Thus we can rewrite (\ref{equation un reecrite}) as
\begin{equation}\label{valeur de gamma enfin}
c=\frac{R}{\pi}(\frac{1}{4T(R)}-\eta)+1
\end{equation}
We now plug $t^2=(c-u^2)T(R)$ and the expression of $c$ given by (\ref{valeur de gamma enfin}) into Equation (2) of (\ref{les cinq identites avec vbis}); we derive
\begin{lem}\label{existence et mega equation}
	Real solutions of the system (\ref{les cinq identites avec vbis} ) exist if and only if the equation
	\begin{equation}\label{la mega inegalite}
	\frac{R}{\pi}(\frac{1}{4T(R)}-\eta)^2+2[1+\frac{R}{\pi}(\frac{1}{4T(R)}-\eta)](\eta-\frac{1}{3}g_2T(R))=2u^2(\eta-\frac{1}{3}g_2T(R)-\frac{2\pi}{R})
	\end{equation}
	has a real solution $u$.
\end{lem}
We will see that the factor of $u^2$ in the (RHS) of (\ref{la mega inegalite}), n is always negative; the bulk of the proof will be to prove that the (LHS) of (\ref{la mega inegalite}) is also negative.

\subsection{The variation of the numerator and denominator of $T$}
\begin{lem}\label{sens de variation}
	1) The numerator of $T$, i.e. 
	$N(R)=R(g_2-12\eta^2)+24\eta\pi$
	is decreasing on $R\in[1,+\infty)]$.\\
	2) The denominator	
	$D(R)=R(12g_3-8g_2\eta)+8g_2\pi$
	is an increasing function of $R\in[1,+\infty)]$.\\
	Thus $T$ is decreasing on $[1,+\infty)$.
\end{lem}
Using the formulae of \S \ref{section avec les series de lang}, we take limits and derive, for $R>1$, 
\begin{equation}\label{limits des d et n}
\begin{split}
 D(1)=0<D(R)<\lim\limits_{R\longrightarrow\infty}D(R)=\dfrac{32\pi^5}{3},\\
 N(1)>N(R)>\lim\limits_{R\longrightarrow\infty}N(R)=8\pi^3
 \end{split}
\end{equation}
{\it Proof of Lemma \ref{sens de variation}}. 
 The expressions in \S \ref{section avec les series de lang} show 
that the functions $\eta$ and $g_3$ are increasing, while $g_2$ is decreasing.  
\begin{equation}\label{serie infinie de derivees}
\mbox{We compute for a given}\ s\in\mathbb{N} \ \ \ \ \ \frac{d}{dR}\sum_{n=1}^\infty \frac{n^sq(R)^n}{1-q(R)^n}=
-2\pi \sum_{n=1}^\infty \frac{n^{s+1}q(R)^n}{(1-q(R)^n)^2}
\end{equation}
Since $q(1)\leq 0.002$, for a positive integer $n$, $1-q(1)^n\geq 1-q(1)>0.99$ hence
\begin{equation}\label{derivee de la serie}
-\frac{2\pi}{0.99} \sum_{n=1}^\infty \frac{n^{s+1}q(R)^n}{1-q(R)^n}\leq\frac{d}{dR}\sum_{i=1}^\infty \frac{n^sq(R)^n}{1-q(R)^n}\leq 
-2\pi \sum_{n=1}^\infty \frac{n^{s+1}q(R)^n}{1-q(R)^n}
\end{equation}
1)	$N'(R)=g_2(R)-12\eta(R)^2+Rg_2'(R)+24\eta'(R)(\pi-R\eta(R))$
	\begin{equation}\label{intermediaire2023} \leq g_2(R)-12\eta(R)^2+Rg_2'(R)
		\end{equation}
	since $\eta$ is increasing and $\eta(1)=\pi$. The constant term of $g_2(R)-12\eta(R)^2$ is zero; using  (\ref{derivee de la serie}), we derive
	$$(\ref{intermediaire2023})
	 < 320\pi^4[\sum_{i=1}^\infty\frac{q(R)^n}{1-q(R)^n}(n^3-2\pi Rn^4)]< 0$$
2)	Since $g_2$ is decreasing, $\eta$ is increasing and $\eta(1)=\pi$,
	\begin{equation}\label{soixante dix}
	D'(R)	\geq  12g_3(R)\underbrace{-8g_2(R)\eta(R)}_{\geq -8g_2(R)\dfrac{\pi^2}{3}}
	+12Rg'_3(R)-8Rg_2(R)\eta'(R)
	\end{equation}

	Since $g_2(1)\leq 190$, we derive, using (\ref{derivee de la serie}),	$$(\ref{soixante dix})\geq \sum_{i=1}^\infty\dfrac{q^n}{1-q^n}[12\pi^6(\frac{8}{27})(504)(2\pi Rn^6-n^5)-8(\frac{4}{9}\pi^6240n^3+2\pi(24)(\frac{\pi}{3})190n^2)] >0.\ \ \ \qed$$
\subsection{The condition for the existence of solutions}
Since $\eta$ is increasing and $T$ and $g_2$ are decreasing,  we have
$$\eta-\frac{g_2 T}{3}\leq \lim_{R\longrightarrow +\infty}
(\eta-\frac{g_2 T}{3})=0.$$ 
Thus the (RHS) in (\ref{la mega inegalite}) is negative, so we derive
\begin{lem}\label{existence et mega equation1}
	A solution exists if and only if 
	\begin{equation}\label{le mega terme}	\underbrace{\frac{R}{\pi}(\frac{1}{4T}-\eta)^2}_{\begin{subarray}{1}\longrightarrow \pi\ \ \mbox{if}\ R\longrightarrow 1 \\
	\longrightarrow 0^+\ \ \mbox{if}\ R>>1\end{subarray}}+2\underbrace{[1+\frac{R}{\pi}(\frac{1}{4T}-\eta)]}_{\begin{subarray}{1}\longrightarrow 0\ \ \mbox{if}\ R\longrightarrow 1 \\
	\longrightarrow 1\ \ \mbox{if}\ R>>1\end{subarray}}\underbrace{(\eta-\frac{1}{3}g_2T)}_{\begin{subarray}{1}\longrightarrow -\infty\ \ \mbox{if}\ R\longrightarrow 1 \\
	\longrightarrow 0^-\ \ \mbox{if}\ R>>1\end{subarray}}\leq 0
	\end{equation}
\end{lem}

\subsection{Proof for $R\geq 1.15$}
If $R\geq 1.15$, we have
$
	q(R)\leq q(1.15)\leq 8.10^{-4}
$. 
The small size of $q$ allows us to approximate $g_2$, $g_3$ and $\eta$ by linear functions in $q$. To this effect, we prove
\begin{lem}
	If $R\geq 1.15$, 
	\begin{enumerate}
		\item 
			$q(R)\leq\underbrace{\sum_{n=1}^\infty \dfrac{nq(R)^n}{1-q(R)^n}}_{(\star)}\leq q(R)(1+0.003)$
		\item 
	$q(R)\leq\underbrace{\sum_{n=1}^\infty \dfrac{n^3q(R)^n}{1-q(R)^n}}_{(\star\star)}\leq q(R)(1+0.007)$
	\item
	$q(R)\leq\underbrace{\sum_{n=1}^\infty \dfrac{n^5q(R)^n}{1-q(R)^n}}_{(\star\star\star)}\leq q(R)(1+0.025)$

		\end{enumerate}
\end{lem}
\begin{proof}
	The LHS inequalities are obvious. To prove the RHS ones,
	\begin{slem}
		If $n$ is a positive integer and $R\geq 1.15$,\\
		1) $nq(R)^n\leq q(R)^{0.9n}$ \ 2) $n^3q(R)^n\leq q(R)^{0.8n}$\ 3) $n^5q(R)^n\leq q(R)^{0.7n}$
	\end{slem} 
	We prove 3); then 1) and 2) are similar.\\
	 We use that $5\dfrac{ln(n)}{n}\leq 5\dfrac{ln(2)}{2}\leq 1.8\leq 0.6\pi 1.15$ and $\dfrac{1}{q-1}=1+\dfrac{q}{q-1}$. 
	 $$(\star\star\star)\leq q(R)\Big[1+\dfrac{q(R)}{1-q(R)}+2^5\dfrac{q(R)}{1-q(R)}
		+\sum_{n=3}^\infty\dfrac{q(R)^{0.7n}}{1-q(R)}\Big]\leq q(R)\Big[1+34\dfrac{q(R)}{1-q(R)}\Big]$$
		$$\leq q(R)\Big[1+34\dfrac{q(1.15)}{1-q(1.15)}\Big]$$
\end{proof}
We derive the inequalities	
\begin{equation}\label{premiere de la liste}
-9\pi^2q\leq \eta-\frac{\pi^2}{3}\leq -8\pi^2q
\end{equation}
\begin{equation}
320\pi^4q\leq g_2-\frac{4\pi^4}{3}\leq 323\pi^4q
\end{equation}
\begin{equation}
-154\pi^6q\leq g_3-\frac{8\pi^6}{27}\leq -149\pi^6q
\end{equation}
\begin{equation}
-6\pi^4q\leq \eta^2-\frac{\pi^4}{9}\leq -5\pi^4q
\end{equation}
\begin{equation}
-3\pi^6q\leq \eta^3-\frac{\pi^6}{27}\leq -2\pi^6q
\end{equation}
\begin{equation}
	92\pi^6q\leq g_2\eta-\frac{4\pi^6}{9}\leq 97\pi^6q
\end{equation}
\begin{equation}
	853\pi^8q\leq g_2^2-\frac{16}{9}\pi^8\leq 938\pi^8q
\end{equation}
\begin{equation}
26\pi^8q\leq g_2\eta^2-\frac{4}{27}\pi^8\leq 30\pi^8q
\end{equation}
\begin{equation}\label{derniere de la liste}
-54\pi^8q\leq g_3\eta-\frac{8}{81}\pi^8\leq -51\pi^8q
\end{equation}

The inequalities (\ref{premiere de la liste})-(\ref{derniere de la liste}) together with $N\geq 8\pi^3$, see (\ref{limits des d et n}) yield estimates for the terms in (\ref{le mega terme}).
\begin{equation}\label{un sur quatre t}
(95-99R\pi)\pi^2q(R)\leq \dfrac{1}{4T(R)}-\eta<0
\end{equation}
Since $R(95-99R\pi)q(R)$ is an increasing function,  we derive 
\begin{equation}
	1+\dfrac{R}{\pi}\big(\dfrac{1}{4T(R)}-\eta\big)>0.3
\end{equation}

\begin{equation}
\mbox{Since}\ D\leq\dfrac{32\pi^5}{3}\ (\mbox{see} (\ref{limits des d et n})),\ \ \ \ \  \eta-\frac{1}{3}g_2T\leq -92R\pi^3 q(R)
\end{equation}
$$(\ref{le mega terme})\leq \dfrac{R}{\pi}[(99R\pi-95)\pi^2q(R)]^2-0.6(92R\pi^3q(R))$$
$$\leq Rq(R)\pi^3\underbrace{[(99R\pi-95)^2q(R)}_{(\star)\leq 50.3}-55.2]<0$$
since $(\star)$ is a decreasing function of $R$ so we estimate it for $R=1.15$.

\subsection{Proof for  $1<R\leq 1.05$}
In this section and the next one, we use numerical estimates for $N$, $D$ $g_2$ etc; we use Sagemath and write the following code.\\
\includegraphics[scale=0.65]{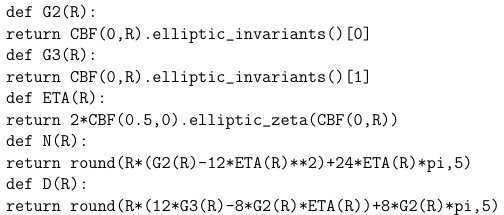}
\subsubsection{Sketch of the proof}
We rewrite (\ref{le mega terme}) as
\begin{equation}\label{expression pour R tout petit}
\underbrace{\frac{R}{\pi}(\frac{D(R)^2}{16N(R)^2}-\eta(R)^2)+2\eta(R)}_{(I)}\ \ \ -\ \ \underbrace{\frac{g_2(R)R}{6\pi}}_{(II)}\ \ \ +\ \ \ \underbrace{\frac{2g_2(R)N(R)}{3\pi}\big(\frac{R\eta(R)-\pi}{D(R)}\big)}_{(III)}
\end{equation}
The term (I) is equal to $\pi$ for $R=1$; we will bound it above by a positive number.\\
The term (II) is negative and we will bound it above by a negative number.\\
In the term (III), both $R\eta-\pi$ and $D$ are zero for $R=1$ so we estimate both $R\eta(R)-\pi$ and $D(R)$ as factors of $R-1$.

\begin{lem} $(I)\leq 3.29$
	\end{lem}
{\it Proof}.
$\dfrac{d}{dR}(-\dfrac{R}{\pi}\eta^2+2\eta)=
-\dfrac{1}{\pi}\eta^2+2\eta'(R)(1-\dfrac{R\eta(R)}{\pi})<0.$

\begin{equation}
\mbox{thus}\ \ \ -\frac{R}{\pi}\eta(R)^2+2\eta(R)\leq -\frac{1}{\pi}\eta(1)^2+2\eta(1)=\pi
\end{equation}
Since $D$ is increasing and $N$ is decreasing, for $1<R<1.05$,\\
$$\dfrac{RD(R)^2}{16\pi N(R)^2}\leq \dfrac{1.05D(1.05)^2}{16\pi N(1.05)^2}\leq 0.14\ \ \ \ \ \ \qed$$
\begin{lem}$(II)\leq -9.1$
	\end{lem}
\begin{proof}
Since $g_2$ is decreasing, 
$(II)\leq -\dfrac{g_2(1.05)}{6\pi}\leq -9.1$. \end{proof}

\begin{lem}\label{lemma iii}$(III)\leq 5.3$.
	\end{lem}
All the factors in (III) are positive. First we notice that
\begin{equation}\label{le gros terme de III}
\frac{2g_2(R)N(R)}{3\pi}\leq \frac{2g_2(1)N(1)}{3\pi}\leq 12338
\end{equation}

\begin{slem}\label{estimee de gtrois pres de un}
	If $R\leq 1.1$, then $g_3(R)\geq 1418(R-1)$
\end{slem}
\begin{proof}

 $g_3(R)=g_3(R)-g_3(1)=(\dfrac{8}{27})504\pi^6\sum_{n=1}^\infty n^5\dfrac{q(1)^n-q(R)^n}{(1-q(1)^n)(1-q(R)^n)}
	$
\begin{equation}\label{g trois}
\geq(\dfrac{8}{27})504\pi^6(q(1)-q(R))=(\dfrac{8}{27})504\pi^6e^{-2\pi}(1-e^{-2\pi(R-1)})
	\end{equation}
To bound (\ref{g trois}) below, we recall: $\forall h>0,\ \ 
	1-e^{-h}\geq h-\dfrac{h^2}{2}$, thus
	\begin{equation} 
	1-e^{-2\pi(R-1)}\geq 2\pi(R-1)[1-\pi(R-1)]\geq 5.29(R-1)
	\end{equation}
\end{proof}
\begin{slem}
	If $R\geq 1$, then  $R\eta-\pi\leq 4.4(R-1)$
\end{slem}
\begin{proof}
	$\eta(R)-\pi = \eta(R)-\eta(1)=8\pi^2\sum_{n=1}^\infty n\dfrac{e^{-2\pi n}
		(1-e^{-2\pi(R-1)n})}{(1-q(1)^n)(1-q(R)^n)}$
	\begin{equation}\label{estimee de eta pres de un}
		\leq (R-1)\dfrac{16\pi^3}{(1-e^{-2\pi})^2}\sum_{n=1}^\infty n^2e^{-2\pi n}
	\end{equation}
	since $1-e^{-2\pi n(R-1)}\leq 2\pi n(R-1)$. Now
	$$\sum_{n=1}^\infty n^2e^{-2\pi n}=e^{-2\pi}+\sum_{n=2}^\infty n^2e^{-2\pi n}\leq e^{-2\pi}+\int_{1}^\infty x^2e^{-2\pi x}dx=e^{-2\pi}(1+\dfrac{1}{2\pi}+\dfrac{1}{2\pi^2}+\dfrac{1}{4\pi^3})$$
	hence $\eta(R)-\pi\leq 1.14(R-1)$ and $R\eta-\pi=R(\eta-\pi)+\pi(R-1)$
\end{proof}
We derive from these Sublemmas,
\begin{equation}\label{lhs pour r pres de 1}
	\frac{D}{R\eta-\pi}\geq 12\Big(\frac{1418}{4.4}\Big)-8g_2(1)\geq 2354.
\end{equation}
Plugging in (\ref{le gros terme de III}), we get 
$
	(III)\leq \dfrac{12338}{2354}\leq 5.3$ which proves Lemma \ref{lemma iii}. \\
	\\
In conclusion, $	(\ref{le mega terme})\leq 3.3-9.1+5.3=-0.5$

\subsection{Proof for  $1.05\leq R\leq 1.15$}
We need to show that (\ref{expression pour R tout petit}) is negative. Since $D$ and $\eta$ are increasing, while $N$ and $g_2$ are decreasing, we have, for $R_1\leq R\leq R_2$, 
\begin{equation}
	(I)\leq \pi+\dfrac{R_2D(R_2)^2}{16N(R_2)^2}=A(R_1,R_2)
	\end{equation}
\begin{equation}
	(II)\leq -\dfrac{R_1g_2(R_2)}{6\pi}=B(R_1,R_2)
\end{equation}
\begin{equation}
	(III)\leq \dfrac{2}{3\pi}[g_2(R_1)N(R_1)\big(\dfrac{R_2\eta(R_2)-\pi}{D(R_1)}\big)]=C(R_1,R_2)
\end{equation}
 $$\mbox{We estimate}\ \ \ A(R_1^{(k)}, R_2^{(k)})+B(R_1^{(k)}, R_2^{(k)})+C(R_1^{(k)}, R_2^{(k)})$$ where  $(R_1^{(k)}=1+\dfrac{k}{100}, R_2^{(k)}=1+\dfrac{k+1}{100})$ where $k=5,6,...,14$ and get \\ 
 $-2.27,
 -2.14,
 -1.96,
 -1.75,
 -1.51,
 -1.25,
 -0.97,
 -0.68,
 -0.38$

\subsection{Explicit formulae}\label{Explicit formulae}
We recall that the map is given  by
\begin{equation}\label{tore minimal - fin}
\Big(\int (-4t^2\wp^2+c\wp+d)+\overline{\int (\wp+s)}, t\wp+\overline{t\wp}+\int (u\wp+v)-\overline{\int (u\wp+v)}\Big).
\end{equation}
Using (\ref{p seconde}), we rewrite it as 
\begin{equation}\label{tore minimal - fin}
\Big( -\dfrac{2}{3}t^2\wp'+\int (c\wp+e)+\overline{\int (\wp+s)}, t\wp+\overline{t\wp}+\int (u\wp+v)-\overline{\int (u\wp+v)}\Big).
\end{equation}
where $e=d-g_2\dfrac{t^2}{3}$.\\
The constants $t,c,d,s,u,v$ are as follows ($T(R)$ is defined in (\ref{definition de T})).
\begin{itemize}

\item 
$u^2=\dfrac{\dfrac{R}{\pi}(\dfrac{1}{4T(R)}-\eta_1)^2+2[1+\dfrac{R}{\pi}(\dfrac{1}{4T(R)}-\eta_1)][\eta_1-\dfrac{1}{3}g_2T(R)]}{2(\eta_1-\frac{1}{3}g_2T(R)-\dfrac{2\pi}{R})}$
\item $c=\dfrac{R}{\pi}(\dfrac{1}{4T(R)}-\eta(R))+1$
\item
$t^2=(c-u^2)T(R)$
\item 
$s=\dfrac{\pi}{R}(c-1)+\eta$
\item
$d=c(\eta_1-\dfrac{\pi}{R})+\dfrac{\pi}{R}+g_2\dfrac{t^2}{3}$
\item
$v=u(\eta_1-\dfrac{2\pi}{R})$
\end{itemize}

\subsection{Exemple for $R=2$}
We round the quantities by $3$ digits. 
\begin{itemize}
	\item 
$u=0.08$
	\item 
$c=0.989$	
	\item 
	$t^2=0.075$
	\item 
	$s=3.272$
	\item 
	$d=6.523$
	\item
$v=0.012$
\end{itemize}

\subsection{Limit cases for the formulae of \S \ref{Explicit formulae}}
\subsubsection{$R\longrightarrow 1^+$}
When $R$ tends to $1^+$, the formulae in (\ref{Explicit formulae}) show us that the tori tend to the $3D$-Chen-Gackstatter torus. Thus we can view the solutions in \S \ref{Explicit formulae} for the rectangular tori as a family which deforms the $3D$-Chen-Gackstatter torus.
\subsubsection{$R\longrightarrow +\infty$}
 The quantity $Re^{-2\pi R}$ converges very fast to $0$, for example, for $R=3$, $e^{-6\pi}$ is of order $10^{-8}$ so the parameters in (\ref{tore minimal - fin}) converge very fast to constant numbers. More  precisely, $u$ and $v$ tend to $0$ and all the other parameters tend to non zero numbers. Thus, the minimal tori (\ref{tore minimal - fin}) are equivalent to 
 $$
 \Big(-\dfrac{\pi^2}{2}\wp'+2Re\big(\int(\wp+\dfrac{\pi^2}{3})\big), \dfrac{\sqrt{3}}{\pi}Re(\int\wp)\Big)
 $$
 \section{The ends of the minimal tori; questions of embeddedness}
 We use the integral formulae of the tangent and normal curvatures given in \S \ref{integral formulae for the normal bundle}.

 If ${\mathcal T}$ is one of the tori constructed above, its total curvature is $-8\pi$, hence $d_++d_-=4$. Moreover ${\mathcal T}$ is not complex (Proposition \ref{pas complexe}), hence neither of the $d_\pm$'s is zero, thus $d_+=d_-=2$ and 
 \begin{equation}\label{fibre normal trivial}
 \int_\Sigma K^N=0
 \end{equation}
 In \S \ref{writhe  at infinity}, we discussed ends and braids at infinity and we now investigate these objects for the minimal tori which we constructed  above. 
 
\subsection{The ends of the tori} 
 
 \begin{enumerate}
 	\item the end of the $3D$ Chen-Gackstatter torus is\\ 
 	\begin{equation}\label{3DCG}
 	z\mapsto(z^3+o(|z^3|),Re(z^2)+o(|z^2|))
 	\end{equation}
 	\item the end of a $4D$ Chen-Gackstatter torus (Theorem \ref{theorem sur le square torus}) is
 	\begin{enumerate}
 		\item if $|\lambda|=1$, 
 		$
 		z\mapsto(z^3+o(|z^3|),Re(z^2)+o(|z^2|),0)
 		$
 		\item if $|\lambda|\neq 1$, after a change of coordinates, \\
 		\begin{equation}\label{4DCG}
 		z\mapsto(z^3+o(|z^3|),z^2+o(|z^2|))
 		\end{equation}
 		The end is the same as the $4D$ Enneper surface with two nodes (\ref{enneper4d points doubles}) and the same as the end of the complex algebraic curves (\ref{algebraic torus}); its braid at infinity is the braid of the $(3,2)$-torus knot, hence $w_\infty=\pm 4$. Thus, (\ref{fibre normal trivial}) together with Theorem  \ref{proposition:curvature of the normal bundle} show that i) these tori are not embedded ii) if one of them has only transverse double points, the total number of these points is $\pm 2$. 		
 			
 \end{enumerate}	
\item for a rectangular torus ${\mathcal T}_R$ (Theorem \ref{thm:rectangular tori}), the end is
\begin{equation}\label{rectangular1}
 	 z\mapsto(z^3+o(|z^3|),Re(z^2)+o(|z^2|),Im(z)+o(|z|))
 	 \end{equation}
 	 If we look at the limit of the knots $\Sigma\cap\mathbb{S}(0,R)$ where $R$ tends to infinity, we get the knot of the $4D$ embedded Enneper surface given by (\ref{embedded4Denneper}) except that it has one singular point.\\
 	 \includegraphics[scale=0.1]{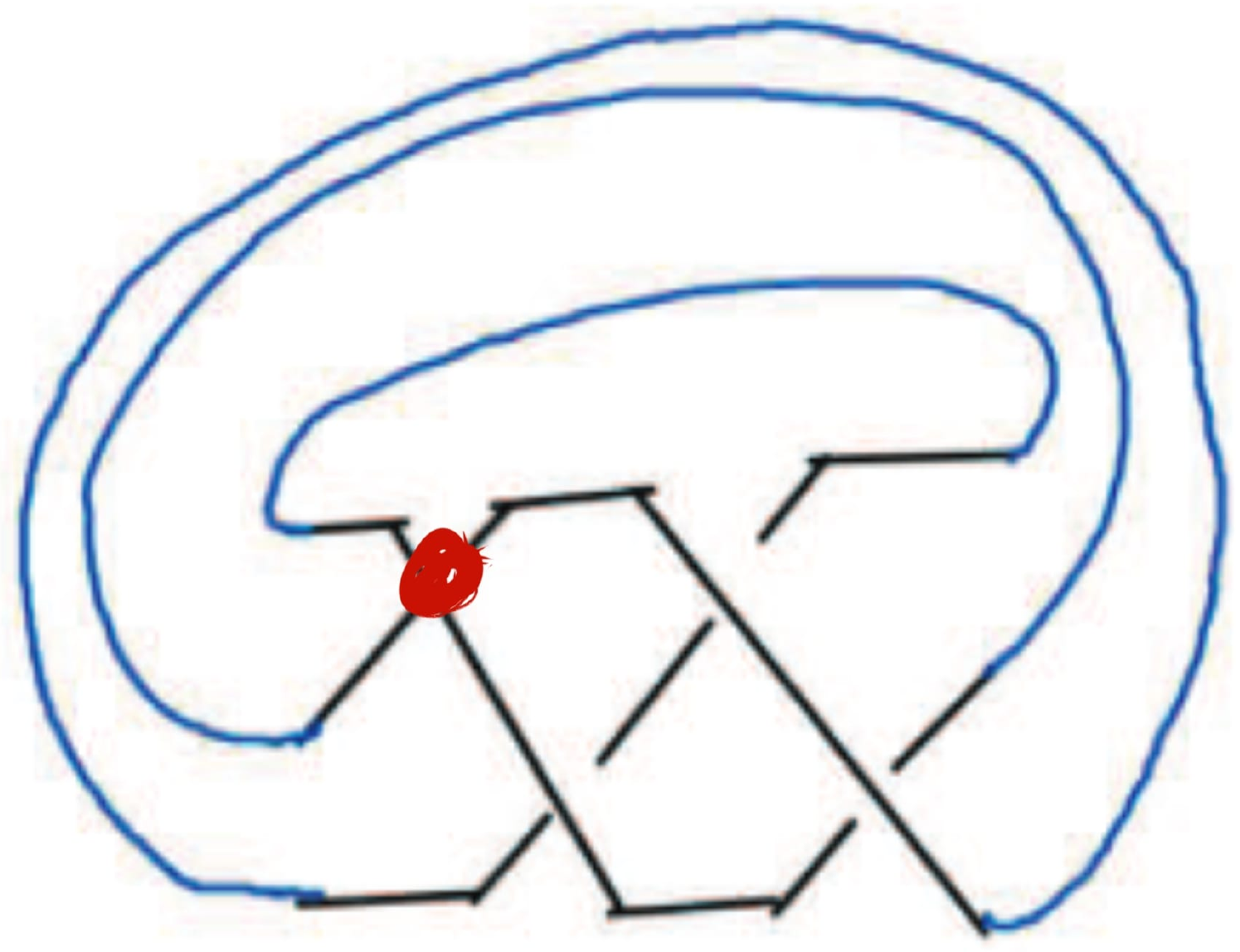}
 	
 \end{enumerate}

Thus we conclude with questions.

\begin{ques}
	Do the tori of Theorem \ref{thm:rectangular tori} have codimension $1$ singularities? Could we get embedded minimal rectangular tori if we look for non real solutions of the equations in Proposition \ref{thm type II 2023}?
	\end{ques}

\footnotesize{Marc Soret: Universit\'e F. Rabelais, D\'ep. de Math\'ematiques, 37000 Tours, France,\\
	Marc.Soret@univ-tours.fr\\
	\\
	Marina Ville: Univ Paris Est Creteil, CNRS, LAMA, F-94010 Creteil, France\\	
	villemarina@yahoo.fr}

\end{document}